\documentclass[11pt]{amsart}
\usepackage{mathptmx, amssymb, amsmath, mathtools, enumerate, colonequals}

\usepackage{xcolor, color}
\usepackage[margin=1.2in]{geometry}
\definecolor{chianti}{rgb}{0.6,0,0}
\definecolor{meretale}{rgb}{0,0,.6}
\definecolor{leaf}{rgb}{0,.35,0}
\usepackage[colorlinks=true, pagebackref, hyperindex, citecolor=meretale, urlcolor=leaf, linkcolor=chianti]{hyperref}

\usepackage[mathscr]{euscript}

\usepackage{tikz,tikz-cd}
\usetikzlibrary{arrows.meta,decorations.markings}
\tikzcdset{
arrow style = tikz,
diagrams={>={Straight Barb[scale=.75]}},
}

\input{xy}
\xyoption{all}

\DeclareFontFamily{OMS}{rsfs}{\skewchar\font'60}
\DeclareFontShape{OMS}{rsfs}{m}{n}{<-5>rsfs5 <5-7>rsfs7 <7->rsfs10 }{}
\DeclareSymbolFont{rsfs}{OMS}{rsfs}{m}{n}
\DeclareSymbolFontAlphabet{\scr}{rsfs}

\usepackage{amsmath,amsfonts,mathtools,amsthm, amssymb}
\usepackage{leftindex}
\usepackage{enumitem}
\usepackage{graphicx,float}

\usepackage[ruled,vlined]{algorithm2e}
\usepackage{algorithmic}

\usepackage{ytableau}
\usepackage{comment}
\usepackage{bigints}
\usepackage{setspace}
\numberwithin{equation}{section}

\newtheorem{theorem}{Theorem}[section]
\newtheorem{lemma}[theorem]{Lemma}
\newtheorem{prop}[theorem]{Proposition}
\newtheorem{corollary}[theorem]{Corollary}

\theoremstyle{definition}
\newtheorem{defn}[theorem]{Definition}
\newtheorem{example}[theorem]{Example}

\theoremstyle{remark}
\newtheorem{remark}[theorem]{Remark}

\newtheorem{question}[theorem]{Question}

\numberwithin{equation}{subsection}

\usepackage [english]{babel}
\usepackage [autostyle, english = american]{csquotes}
\MakeOuterQuote{"}

\newcommand{\frakp}{\mathfrak{p}}

\newcommand{\RR}{\mathbb{R}}
\newcommand{\ZZ}{\mathbb{Z}}

\DeclareMathOperator{\Spec}{Spec}
\DeclareMathOperator{\Supp}{Supp}

\DeclareMathOperator{\bF}{\bf{F}}

\DeclareMathOperator{\coker}{coker}
\DeclareMathOperator{\im}{image}

\makeatletter
\@namedef{subjclassname@2020}{%
  \textup{2020} Mathematics Subject Classification}
\makeatother

\title{Koszul cohomology and support of local cohomology modules of complete intersections}
\author{Michael Gintz and Wenliang Zhang}
\address{Department of Mathematics, Statistics, and Computer Science, University of Illinois at Chicago,
Chicago, IL 60607}
\email{mgintz2@uic.edu, wlzhang@uic.edu}
\thanks{The authors acknowledge support from NSF through the grant DMS-1752081.}

\begin{document}

\maketitle

\begin{abstract}
Let $R$ be a noetherian commutative ring and $\underline{f}\in R$ be a regular sequence. We introduce a framework to study $\Supp(H^j_I(R/(\underline{f})))$ by linking the Koszul cohomology of $H^j_I(R)$ on the regular sequence $\underline{f}$ and local cohomology modules $H^j_I(R/(\underline{f}))$. As an application, we prove that if $R$ is a noetherian regular ring of prime characteristic $p$ and $f_1,f_2\in R$ form a regular sequence then $\Supp(H^j_I(R/(f_1,f_2)))$ is Zariski-closed for each integer $j$ and each ideal $I$. 
\end{abstract}


\section{Introduction}
Let $R$ be a noetherian commutative ring and $I$ be an ideal. Let $\Gamma_I$ denote the $I$-torsion functor defined via:
\[\Gamma_I(M)=\{z\in M\mid I^tz=0\ {\rm for\ some\ integer\ }t\};\quad \Gamma_I(M\xrightarrow{f}N)=\Gamma_I(M)\xrightarrow{f|_{\Gamma_I(M)}}\Gamma_I(N).\]
It turns out that $\Gamma_I$ is left-exact; the $j$-th local cohomology of an $R$-module $M$, denoted by $H^j_I(M)$, is defined as $\RR^j\Gamma_I(M)$; that is
\[H^j_I(M)\cong H^j(0\to Q^\bullet)\]
where $0\to M\to Q^{\bullet}$ is an injective resolution of $M$. It can be calculated by a \v{C}ech complex; {\it cf.} \S\ref{double complex and spectral sequence} for details.

Since the theory of local cohomology was introduced in \cite{SGA2}, the study of finiteness properties of these modules, as well as their vanishing, has become an active research topic. The interested reader is referred to \cite{HunekeLCProblems} for a list of inspiring open questions on vanishing and finiteness properties of local cohomology modules. One of these question asks whether the set of associated primes of $H^j_I(R)$ is finite for each integer $j$ and each ideal $I$ in $R$. Some positive answers are known: when $R$ is a regular ring of equi-characteristic $p$ (\cite{HunekeSharp}), when $R$ is either a regular local ring of equi-characteristic 0 or a regular affine ring of equi-characteristic 0 (\cite{Lyubeznik1993}), when $R$ is an unramified regular local ring of mixed characteristic (\cite{LyubeznikUnramified}), when $R$ is a smooth $\ZZ$-algebra, and when either $\dim(R)$ or $j$ is sufficiently small ({\it cf.} \cite{KS1999, BRS00, Hellus2001}). Examples in \cite{SinghExample2000, Katzman2002, SinghSwanson} show that local cohomology modules may have infinitely many associated primes. However, the following question ({\it cf.} \cite[p.~3194]{HunekeKatzMarley}) remains open:

\begin{question}
\label{ques: closed}
Let $R$ be a noetherian commutative ring and $I$ be an ideal. Is $\Supp(H^j_I(R))$ Zariski-closed in $\Spec(R)$ for each integer $j$? 
\end{question}

Note that $\Supp(H^j_I(R))$ being Zariski-closed is equivalent to having finitely many \emph{minimal} associated primes. Hence Question \ref{ques: closed} concerns with a finiteness property of local cohomology modules. \cite[p.~3195]{HunekeKatzMarley} states that ``Clearly, this question is of central
importance in the study of cohomological dimension and understanding the local–global properties of local cohomology.'' Some positive answers to Question \ref{ques: closed} are known: when $j=2$ and $H^t_I(R)=0$ for all $t>2$ ( \cite[Theorem~1.2]{HunekeKatzMarley}) an when $R=S/(f)$ where $S$ is a noetherian regular ring of prime characteristic $p$ (\cite{HochsterNB, KatzmanZhang}).

One of the main results of this article is the following:
\begin{theorem}[=Theorem \ref{closed for lc of ci}]
\label{main thm on closedness}
Let $S$ be a noetherian regular ring of prime characteristic $p$ and $f_1,f_2$ be a regular sequence in $S$. Set $R=S/(f_1,f_2)$. Then $\Supp(H^j_I(R))$ is Zariski-closed for each integer $j$ and each ideal $I$.
\end{theorem}

Our strategy to prove Theorem \ref{main thm on closedness} is to link the Koszul cohomology groups of $H^j_I(R)$ on a sequence $\underline{f}$ to the local cohomology modules $H^i_I(R/(\underline{f}))$ via a double complex. To wit, let $R$ be a noetherian ring and $\underline{f}=f_1,\dots,f_c$ be a sequence of elements. Let $I=(g_1,\dots,g_t)$ be an ideal in $R$. Let $\check{C}^\bullet(\underline{g};N)$ denote the \v{C}ech complex of an $R$-module $N$ on the sequence $\underline{g}$ and let $K^{\bullet}(\underline{f};N)$ denote the Koszul (co)complex of an $R$-module $N$ on the sequence $\underline{f}$. Let ${\bf D}$ denote the double complex whose $i$-the row is the \v{C}ech complex $\check{C}^{\bullet}(\underline{g};K^i(\underline{f};R))$ and whose $j$-th column is the the Koszul (co)complex $K^{\bullet}(\underline{f}; C^j(\underline{g};R))$. Then there is a spectral sequence
\[E^{i,j}_2:=H^i(K^{\bullet}(\underline{f}; H^j_I(R))\Rightarrow H^{i+j}(T^\bullet)\]
associated with ${\bf D}$, where $T^\bullet$ denotes the total complex of ${\bf D}$ ({\it cf.} \S\ref{double complex and spectral sequence} for details). The following theorem provides a framework to study $\Supp(H^k_I(R/(\underline{f})))$ via investigating $H^i(K^{\bullet}(\underline{f}; H^j_I(R))$.

\begin{theorem}[=Theorem \ref{closedness in one spectral implies the other}]
\label{thm: intro spectral}
Let $R$ be a noetherian ring, $I=(g_1,\dots,g_t)$ be an ideal, and $f_1,\dots,f_c$ be a sequence of elements in $R$. Let $E^{\bullet,\bullet}_2$ be as above. Assume that
\begin{enumerate}
\item $\Supp(E^{i,j}_{\infty})$ are Zariski-closed for all integers $i,j$, and that
\item $f_1,\dots,f_c$ form a regular sequence in $R$.
\end{enumerate}
Then $\Supp(H^k_I(R/(f_1,\dots,f_c)))$ is Zariski-closed for each integer $k$.
\end{theorem}  

This article is organized a follows. In \S\ref{double complex and spectral sequence}, we introduce and study a double complex which links the Koszul cohomology of $H^j_I(R)$ on a sequence $\underline{f}$ and the local cohomology modules $H^j_I(R/(\underline{f}))$ and prove Theorem \ref{thm: intro spectral}; \S\ref{double complex and spectral sequence} is characteristic-free and does not require $R$ to be regular. In \S\ref{truncated Cech complex}, we introduce the notion of the (Frobenius) truncation of \v{C}ech complexes which is one of the main technical tools in this article. In \S\ref{Koszul of F-modules} and \S\ref{1st Koszul length 2}, we prove that $H^i(K^\bullet(f_1,f_2;\scr{M}))$ has Zariski-closed support when $f_1,f_2$ form a regular sequence in regular ring $R$ of prime characteristic $p$ and $\scr{M}$ is an $F$-finite $F$-module. In \S\ref{support of E infinity}, we complete the proof of Theorem \ref{main thm on closedness}.
\section{A Koszul-\v{C}ech double complex and related spectral sequences}
\label{double complex and spectral sequence}
Let $R$ be a commutative noetherian ring and $f_1,\dots,f_c$ and $g_1,\dots,g_t$ be two sequences of elements in $R$. Set $I=(g_1,\dots,g_t)$ to be the ideal generated by $g_1,\dots,g_t$. For each $R$-module $N$, 
\begin{enumerate}
\item we denote by $K^{\bullet}(\underline{f};N)$ the Koszul co-complex of $N$ on the elements $f_1,\dots,f_c$, which is the $R$-dual of the Koszul complex $K_{\bullet}(\underline{f};N)$,  and
\item we denote by $\check{C}^{\bullet}(\underline{g};N)$ the $\check{C}$ech complex of $N$ on $g_1,\dots,g_t$:
 \[0\to N\xrightarrow{\delta^0} \oplus_{i=1}^tN_{g_i}\xrightarrow{\delta^1} \oplus_{i_1<i_2}N_{g_{i_1}g_{i_2}}\xrightarrow{\delta^2}\cdots \to N_{g_1\cdots g_t}\to 0,\]
 where $\delta^i$ is defined via $\delta^i: N_{g_{j_1}\cdots g_{j_i}}\to N_{g_{\ell_1}\cdots g_{\ell_{i+1}}}$ is defined as 
\[\delta^i(\frac{z}{g^n_{j_1}\cdots g^n_{j_i}})=\begin{cases} (-1)^{s-1}\frac{z}{g^n_{j_1}\cdots g^n_{j_i}} & {\rm when\ }j_1\cdots j_i=\ell_{1}\cdots \hat{\ell}_s\cdots \ell_{i+1}\\ 0&{\rm otherwise}\end{cases}\]
for each $z\in N$. Note that $H^j(\check{C}^{\bullet}(\underline{g};N))\cong H^j_I(N)$.
 \end{enumerate}

\begin{defn}
\label{defn: Cech-Koszul double}
The double complex, denoted by 
${\bf D}:=D(K^{\bullet}(\underline{f});\check{C}^{\bullet}(\underline{g}))$ 
is the double complex complex whose $i$-the row is the \v{C}ech complex $\check{C}^{\bullet}(\underline{g};K^i(\underline{f};R))$ and whose $j$-th column is the the Koszul (co)complex $K^{\bullet}(\underline{f}; C^j(\underline{g};R))$.

We will denote the total complex of ${\bf D}$ by $T^{\bullet}$.
\end{defn}

\begin{example}[When $t=2$]
The most relevant case for this article is when $t=2$ and we would like to spell out the double complex as follows. The Koszul (co)complex $K^{\bullet}(f_1,f_2;N)$ is the following for each $R$-module $N$:
\[0\to N\xrightarrow{{\small \begin{pmatrix}f_1\\f_2 \end{pmatrix}}} N^{\oplus 2}\xrightarrow{{\small \begin{pmatrix}-f_2& f_1 \end{pmatrix}}}N\to 0\]

The Koszul-\v{C}ech double complex in this case is the following:
\begin{equation}
\label{KC double complex}
\xymatrix{
& 0 & 0& 0 & & 0 \\
	0\ar[r] & R \ar[r]\ar[u] &{\bigoplus_{j} R_{g_j}} \ar[r]\ar[u] &  {\bigoplus_{j_1<j_2} R_{g_{j_1}g_{j_2}}} \ar[r]\ar[u] & \cdots \ar[r] & {R_{g_{1}\cdots g_{t}}}\ar[r]\ar[u] & 0 \ar[l] \\
	0\ar[r] & R^{\oplus 2} \ar[r]\ar[u]^{{\tiny \begin{pmatrix}-f_2& f_1 \end{pmatrix}}} & {(\bigoplus_{j} R_{g_{j}})^{\oplus 2}} \ar[r]\ar[u]^{{\tiny \begin{pmatrix}-f_2& f_1 \end{pmatrix}}} & {(\bigoplus_{j_1<j_2} R_{g_{j_1}g_{j_2}})^{\oplus 2}} \ar[r]\ar[u]^{{\tiny \begin{pmatrix}-f_2& f_1 \end{pmatrix}}}  & \cdots \ar[r]  & {(R_{g_{1}\cdots g_{t}})^{\oplus 2}} \ar[r]\ar[u]^{{\tiny \begin{pmatrix}-f_2& f_1 \end{pmatrix}}}  & 0 \\
	0\ar[r] & R  \ar[r]\ar[u]^{{\tiny \begin{pmatrix}f_1\\f_2 \end{pmatrix}}} & {\bigoplus_{j} R_{g_j}} \ar[r]\ar[u]^{{\tiny \begin{pmatrix}f_1\\f_2 \end{pmatrix}}} & {\bigoplus_{j_1<j_2} R_{g_{j_1}g_{j_2}}} \ar[r]\ar[u]^{{\tiny \begin{pmatrix}f_1\\f_2 \end{pmatrix}}} & \cdots\ar[r] & {R_{g_{1}\cdots g_{t}}} \ar[r]\ar[u]^{{\tiny \begin{pmatrix}f_1\\f_2 \end{pmatrix}}} & 0 \\
	& 0\ar[u] & 0\ar[u]& 0\ar[u] & & 0\ar[u]
}
\end{equation}
\end{example}

\begin{remark}
\label{spectral from double}
As discussed in \cite[\S5.1]{WeibelBook}, there are two spectral sequences associated with our complex $D(K^{\bullet}(\underline{f});\check{C}^{\bullet}(\underline{g}))$. 

One of the them comes from taking horizontal differentials (in the \v{C}ech complexes) first and then vertical differentials (in the resulted Koszul co-complexes). The resulted spectral sequence is:
\[E^{i,j}_2:=H^i(K^{\bullet}(\underline{f}; H^j_I(R))\Rightarrow H^{i+j}(T^\bullet)\]
Recall that $T^\bullet$ is the total complex of $D(K^{\bullet}(\underline{f});\check{C}^{\bullet}(\underline{g}))$.

 The other one comes from doing differentials the other way around (considering vertical differentials and then horizontal differentials):
 \[{'E}^{i,j}_2:=H^i_I(H^j(K^{\bullet}(\underline{f};R))\Rightarrow H^{i+j}(T^{\bullet})\]
\end{remark}

The following theorem, one of our main technical tools, indicates the connection between $\Supp(E^{i,j}_{\infty})$ and $\Supp(H^k_I(R/(f_1,\dots,f_s)))$ when $f_1,\dots,f_s$ form a regular sequence in $R$.

\begin{theorem}
\label{closedness in one spectral implies the other}
Assume that
\begin{enumerate}
\item $\Supp(E^{i,j}_{\infty})$ are Zariski-closed for all integers $i,j$, and that
\item $f_1,\dots,f_c$ form a regular sequence in $R$.
\end{enumerate}
Then $\Supp(H^k_I(R/(f_1,\dots,f_c)))$ is Zariski-closed for each integer $k$.
\end{theorem}
\begin{proof}
The convergence
\[E^{i,j}_2:=H^i(K^{\bullet}(\underline{f}; H^j_I(R))\Rightarrow H^{i+j}(T^{\bullet})\]
amounts to a filtration of $H^k(T^\bullet)$ for each $k$:
\[0\subseteq F^kH^k(T^{\bullet})\subseteq F^{k-1}H^k(T^{\bullet})\subseteq \cdots\subseteq F^1H^k(T^{\bullet})\subseteq F^0H^k(T^{\bullet})=H^k(T^{\bullet})\]
such that $F^iH^k(T^{\bullet})/F^{i+1}H^{k}(T^\bullet)\cong E^{i,n-i}_{\infty}$ (with $F^kH^k(T^\bullet)\cong E^{k,0}_\infty$). 

Since $E^{i,j}_{\infty}$ is Zariski closed for all integers $i,j$, the Zariski-closedness of $\Supp(H^k(T^\bullet))$ follows from the filtration of $H^k(T^\bullet)$.

The assumption that $f_1,\dots,f_s$ form a regular sequence in $R$ implies that ${'E}^{\bullet,\bullet}_2$ has only one nonzero row in which the entries are $H^i_I(R/(f_1,\dots,f_c))$. Consequently $H^k_I(R/(f_1,\dots,f_c))\cong H^k(T^{\bullet})$ which shows that $\Supp(H^k_I(R/(f_1,\dots,f_c))$ is Zariski closed.
\end{proof}

In \S\ref{support of E infinity}, we will prove that $\Supp(E^{i,j}_{\infty})$ are Zariski-closed for all integers $i,j$ when $R$ is \emph{regular} of prime characteristic $p$ and $E^{i,j}_{\infty}$ are associated with the double complex (\ref{KC double complex}). One of our technical tools is to truncate the \v{C}ech complex.

\section{Truncated Cech complexes}
\label{truncated Cech complex}

In this section we explain truncated \v{C}ech complexes, one of the main technic tools needed in this article.

Let $R$ be a Noetherian commutative ring of prime characteristic $p>0$ and let $g\in R$ be an element in $R$. We will use $R\cdot \frac{1}{g^{p^e}}$ denote the cyclic $R$-submodule of $R_f$ generated by $\frac{1}{g^{p^e}}$, and we will call $R\cdot \frac{1}{g^{p^e}}$ the $e$-th (Frobenius) truncation of $R_g$. (Our convention is to consider $R\cdot \frac{1}{g}$ as the $0$-th Frobenius truncation of $R_g$.)

Note that $R\cdot \frac{1}{g^{p^e}}$ is a finitely generated $R$-module; this finiteness plays a crucial role in this article. 

\begin{remark}
Let $g_1,\dots,g_t$ be elements in $R$. Recall that $\check{C}^\bullet(\underline{g};R)$, the \v{C}ech complex of $R$ on $g_1,\dots,g_t$, is constructed as follows:
\[0\to R\to \bigoplus_{j=1}^tR_{g_j}\to  \cdots \to \bigoplus_{j_1<\cdots<j_i}R_{g_{j_1}\cdots g_{j_i}}\xrightarrow{\delta^i}\bigoplus_{j_1<\cdots<j_{i+1}}R_{g_{j_1}\cdots g_{j_{i+1}}}\to \cdots\to R_{g_1\cdots g_t}\to 0\]
where $\delta^i$ is defined via $\delta^i: R_{g_{j_1}\cdots g_{j_i}}\to R_{g_{\ell_1}\cdots g_{\ell_{i+1}}}$ is defined as 
\begin{equation}
\label{equ: i-th diff in Cech}
\delta^i(\frac{r}{g^n_{j_1}\cdots g^n_{j_i}})=\begin{cases} (-1)^{s-1}\frac{r}{g^n_{j_1}\cdots g^n_{j_i}} & {\rm when\ }j_1\cdots j_i=\ell_{1}\cdots \hat{\ell}_s\cdots \ell_{i+1}\\ 0&{\rm otherwise}\end{cases}
\end{equation}

Then it is clear that the image of the restriction of $\delta^i$ on $R\cdot \frac{1}{g^{p^e}_{j_1}\cdots g^{p^e}_{j_i}}$ is contained in $R\cdot \frac{1}{g^{p^e}_{\ell_1}\cdots g^{p^e}_{\ell_{i+1}}}$. Consequently, if one replaces each module in the Cech complex $C^\bullet(\underline{g};R)$ by its $e$-th truncation, then one will get a complex
\begin{equation}
\label{equ: e-th truncation of cech}
0\to R\to \bigoplus_{j=1}^tR\cdot \frac{1}{g^{p^e}_j}\to  \bigoplus_{j_1<j_2}R\cdot \frac{1}{g^{p^e}_{j_1}g^{p^e}_{j_2}}\to\cdots
\end{equation}
\end{remark}

\begin{defn}
The complex (\ref{equ: e-th truncation of cech}) is called the $e$-th truncation of the \v{C}ech complex $\check{C}^\bullet(\underline{g};R)$ and will be denoted by $\check{C}^\bullet(\underline{g};R)_e$ or $\check{C}^\bullet_e$ when the elements $g_1,\dots,g_t$ are clear from the context. The $i$-th term in $\check{C}^\bullet(\underline{g};R)_e$ will be denoted by $\check{C}^i(\underline{g};R)_e$ and the $i$-th differential in $\check{C}^\bullet(\underline{g};R)_e$ will be denoted by $\delta^i_e$.
\end{defn}

For each element $\eta\in \ker(\delta^i)$ (respectively $\eta\in \ker(\delta^i_e)$), its image in $H^i(\check{C}^\bullet(\underline{g};R))$ (respectively $H^i(\check{C}^\bullet(\underline{g};R)_e)$) will be denoted by $[\eta]$.

Let $R$ be a noetherian ring of prime characteristic $p$. Let $R^{(e)}$ be the additive group of $R$ regarded as an $R$-bimodule with the usual left $R$-action and with the right $R$-action defined by $r'r= r^{p^e}r'$ for all $r\in R$ and $r'\in R^{(e)}$. The $e$-th Peskine-Szipro functor $\bF^e$ is defined via
\[\bF(M)=R^{(e)}\otimes_RM\quad \bF(M\xrightarrow{\phi} N)=R^{(e)}\otimes_RM\xrightarrow{{\bf 1}\otimes \phi}R^{(e)}\otimes_RN.\]
When $e=1$, we will denote $\bF^1$ by $\bF$. 

Note that, when $R$ is regular, $R^{(e)}$ is a faithfully flat $R$-module and hence $\bF^e$ is an exact functor for each $e\geq 1$ (\cite{Kunz1969}).

\begin{prop}
\label{apply Frob to truncated Cech}
Let $R$ be a Noetherian regular ring of prime characteristic $p>0$ and let $\bF$ denote the Peskine-Szpiro functor. Then
\begin{enumerate}
\item $\bF(R\cdot \frac{1}{g})\cong R\cdot \frac{1}{g^{p}}$ for every $g\in R$.
\item $\bF (\check{C}^\bullet(\underline{g};R)_e)\cong \check{C}^\bullet(\underline{g};R)_{e+1}$ for all sequences of elements $\underline{g}=g_1,\dots,g_t$.
\end{enumerate}
\end{prop}
\begin{proof}
Note that $\bF$ is an exact functor since $R$ is regular. 

To prove the first part, it suffices to note that the $R$ linear map 
\[\theta: \bF(R\cdot \frac{1}{g})=R^{(1)}\otimes_RR\cdot \frac{1}{g}\xrightarrow{r'\otimes \frac{r}{g}\mapsto \frac{r'r^p}{g^p}}R\cdot \frac{1}{g^p}\]
admits an inverse
\[R\cdot \frac{1}{g^p}\xrightarrow{\frac{r}{g^p}\mapsto r\otimes\frac{1}{g}} R^{(1)}\otimes_RR\cdot \frac{1}{g}=\bF(R\cdot \frac{1}{g}).\]

The second part follows from the following commutative diagram
\[\xymatrix{
\bF(R\cdot \frac{1}{g^{p^e}_{j_1}\cdots g^{p^e}_{j_i}}) \ar[r] \ar[d] & \bF(R\cdot \frac{1}{g^{p^e}_{\ell_1}\cdots g^{p^e}_{\ell_{i+1}}}) \ar[d]\\
R\cdot \frac{1}{g^{p^{e+1}}_{j_1}\cdots g^{p^{e+1}}_{j_i}} \ar[r] & R\cdot \frac{1}{g^{p^{e+1}}_{\ell_1}\cdots g^{p^{e+1}}_{\ell_{i+1}}}
}\]
where the horizontal maps are induced by the $i$-th differential (\ref{equ: i-th diff in Cech}) in the \v{C}ech complex and the vertical maps are the isomorphisms in the first part applied to the cases when $g=g^{p^e}_{j_1}\cdots g^{p^e}_{j_i}$ and when $g=g^{p^e}_{\ell_1}\cdots g^{p^e}_{\ell_{i+1}}$, respectively.
\end{proof}

For the rest of this article, we will denote by $\theta$ the isomorphisms 
\[\bF^e(\check{C}^j(\underline{g};R))\xrightarrow{\sim} \check{C}^j(\underline{g};R),\quad \bF^e(\check{C}^j_e)\xrightarrow{\sim} \check{C}^j_{e+1}\quad {\rm and}\quad \bF^e(\check{C}^j_0)\xrightarrow{\sim} \check{C}^j_e.\]

The natural inclusion $R\cdot \frac{1}{g^{p^e}_{j_1}\cdots g^{p^e}_{j_i}}\to R\cdot \frac{1}{g^{p^{e+1}}_{j_1}\cdots g^{p^{e+1}}_{j_i}}$ induces a chain map between the truncated \v{C}ech complexes: $\check{C}^\bullet(\underline{g};R)_e\to \check{C}^\bullet(\underline{g};R)_{e+1}$ and hence induces an $R$-module homomorphism $H^i(\check{C}^\bullet(\underline{g};R)_e)\to H^i(\check{C}^\bullet(\underline{g};R)_{e+1})$. This produces a directed system:
\[H^i(\check{C}^\bullet(\underline{g};R)_0)\to H^i(\check{C}^\bullet(\underline{g};R)_1)\to \cdots \to H^i(\check{C}^\bullet(\underline{g};R)_e)\to \cdots\]
whose direct limit is isomorphic to $H^j_I(R)$.

Each element in $H^i_I(R)$ can be represented by a cohomological class of the form $[\cdots,\frac{r}{g^{n}_{j_1}\cdots g^{n}_{j_i}},\cdots]$. Let $H^j_I(R)_e$ be the $R$-submodule of $H^j_I(R)$ generated by classes $[\cdots,\frac{r}{g^{n}_{j_1}\cdots g^{n}_{j_i}},\cdots]$ with $n\leq p^e$. Then $H^j_I(R)_e$ is precisely the image of $H^i(\check{C}^\bullet(\underline{g};R)_e)$ in $H^j_I(R)$; consequently $H^i_I(R)_e$ is finitely generated. Furthermore, one can check that
\begin{equation}
\label{image of truncated Cech in LC}
H^i_I(R)_e\cong \frac{\ker(\delta^i_e)}{\im(\delta^{i-1})\cap \ker(\delta^i_e)}\ {\rm and\ } \bF(H^i_I(R)_e)\cong H^i_I(R)_{e+1}.
\end{equation}

For the rest of this article, whenever it is clear from the context, we will write $\check{C}^{\bullet}(\underline{g})$, or even $\check{C}^{\bullet}$, instead of $\check{C}^{\bullet}(\underline{g};R)$.

One can replace the \v{C}ech complex with its (Frobenius) truncations in Definition \ref{defn: Cech-Koszul double} to form the double complex 
\[{\bf D}_e:=D(K^{\bullet}(\underline{f^{p^e}});\check{C}^{\bullet}(\underline{g})_e)\] 
for each integer $e\geq 0$: 
\begin{equation}
\label{apply Frob to double with truncated}
\xymatrix{
& 0 & 0& 0 & & 0 \\
	0\ar[r] & R \ar[r]\ar[u] &{\bigoplus_{j} R\cdot\frac{1}{g^{p^e}_j}} \ar[r]\ar[u] &  {\bigoplus_{j_1<j_2} R\cdot\frac{1}{g^{p^e}_{j_1}g^{p^e}_{j_2}}} \ar[r]\ar[u] & \cdots \ar[r] & {R\cdot\frac{1}{g^{p^e}_{1}\cdots g^{p^e}_{s}}}\ar[r]\ar[u] & 0 \ar[l] \\
	0\ar[r] & R^{\oplus 2} \ar[r]\ar[u]^{{\tiny \begin{pmatrix}-f^{p^e}_2& f^{p^e}_1 \end{pmatrix}}} & {(\bigoplus_{j} R\cdot\frac{1}{g^{p^e}_j})^{\oplus 2}} \ar[r]\ar[u]^{{\tiny \begin{pmatrix}-f^{p^e}_2& f^{p^e}_1 \end{pmatrix}}} & {(\bigoplus_{j_1<j_2} R\cdot\frac{1}{g^{p^e}_{j_1}g^{p^e}_{j_2}})^{\oplus 2}} \ar[r]\ar[u]^{{\tiny \begin{pmatrix}-f^{p^e}_2& f^{p^e}_1 \end{pmatrix}}}  & \cdots \ar[r]  & {(R\cdot\frac{1}{g^{p^e}_{1}\cdots g^{p^e}_{s}})^{\oplus 2}} \ar[r]\ar[u]^{{\tiny \begin{pmatrix}-f^{p^e}_2& f^{p^e}_1 \end{pmatrix}}}  & 0 \\
	0\ar[r] & R  \ar[r]\ar[u]^{{\tiny \begin{pmatrix}f^{p^e}_1\\f^{p^e}_2 \end{pmatrix}}} & {\bigoplus_{j} R\cdot\frac{1}{g^{p^e}_j}} \ar[r]\ar[u]^{{\tiny \begin{pmatrix}f^{p^e}_1\\f^{p^e}_2 \end{pmatrix}}} & {\bigoplus_{j_1<j_2} R\cdot\frac{1}{g^{p^e}_{j_1}g^{p^e}_{j_2}}} \ar[r]\ar[u]^{{\tiny \begin{pmatrix}f^{p^e}_1\\f^{p^e}_2 \end{pmatrix}}} & \cdots\ar[r] & {R\cdot\frac{1}{g^{p^e}_{1}\cdots g^{p^e}_{s}}} \ar[r]\ar[u]^{{\tiny\begin{pmatrix}f^{p^e}_1\\f^{p^e}_2 \end{pmatrix}}} & 0 \\
	& 0\ar[u] & 0\ar[u]& 0\ar[u] & & 0\ar[u]
}
\end{equation}

A priori, one can form the double complex $D(K^{\bullet}(\underline{f^{p^{e}}});\check{C}^{\bullet}(\underline{g})_{e'})$ for two different integers $e$ and $e'$. Since this is not needed in this article, we opt not to explore it here.

We will denote the total complex of (\ref{apply Frob to double with truncated}) by $T^\bullet_e$. When taking the horizontal differentials (those in the truncated \v{C}ech complexes) and then the vertical differentials in (\ref{apply Frob to double with truncated}), one obtains a spectral sequence:
\begin{equation}
\label{spectral sequence truncated}
E^{i,j}_{2,e}:=H^i(K^{\bullet}(\underline{f}; H^j(\check{C}_0)\Rightarrow H^{i+j}(T^\bullet_e)
\end{equation}
We will denote the differentials in (\ref{spectral sequence truncated}) by
\[\varphi^{i,j}_{2,e}:E^{i,j}_{2,e}\to E^{i+2,j-1}_{2,e}.\]

Since $\bF$ is an exact functor, one can check $\bF^e(K^{\bullet}(\underline{f};R))\cong K^{\bullet}(\underline{f}^{p^p};R)$ for any sequence $\underline{f}$ of elements in $R$. On the other hand, according to Proposition \ref{apply Frob to truncated Cech} that $\bF^e(\check{C}^{\bullet}(\underline{g})_0)\cong \check{C}^{\bullet}(\underline{g})_e$ for any sequence $\underline{g}$ of elements in $R$. Consequently, the double complex ${\bf D}_e$ can be obtained by applying $\bF^e$ to ${\bf D}_0$.

According to Theorem \ref{closedness in one spectral implies the other}, it suffices to analyze the double complex ${\bf D}$. One of our motivations to introduce the double complexes ${\bf D}_e$ is that a great deal of information of ${\bf D}$ is already encoded in ${\bf D}_0$ in which every module is finitely generated. As shown in the sequel, one can link ${\bf D}_0$ with ${\bf D}$ using the Peskine-Szpiro functor $\bF$. This link is rather intricate since ${\bf D}_0$ is directly linked with ${\bf D}_e$ via $\bF^e$  (the differentials in the Koszul (co)complex in ${\bf D}_e$ come from the elements $f^{p^e}_1,f^{p^e}_2$, not $f_1,f_2$).  
\section{Koszul cohomology of $F$-finite $F$-modules}
\label{Koszul of F-modules}

Let $R$ be a noetherian \emph{regular} ring of prime characteristic $p>0$. In this section, we will investigate $E^{i,j}_2$ in the $E^{\bullet,\bullet}_2$-page coming from the double complex ${\bf D}$ has Zariski-closed support; that is the Koszul cohomology $H^i(K^\bullet(\underline{f};H^j_I(R)))$. Instead of local cohomology modules $H^j_I(R)$, we will consider all $F$-finite $F$-modules. To this end, we begin by recalling the definition and basic facts of $F$-modules ({\it cf.} \cite{LyubeznikFModule}).

\begin{enumerate}
\item An $R$-module $\scr{M}$ is an \emph{$F$-module} if there is an $R$-module isomorphism
\[\theta: \scr{M}\to \bF(\scr{M})=R^{(1)}\otimes_R\scr{M}\]
called the structure isomorphism.

\item If $(\scr{M},\theta_\scr{M})$ and $(\scr{N},\theta_\scr{N})$ are $F$-modules, then an \emph{$F$-module morphism} from $(\scr{M},\theta_\scr{M})$ to $(\scr{N},\theta_\scr{N})$ consists of the the following commutative diagram:
\[\xymatrix{
\scr{M} \ar[r]^{\varphi} \ar[d]^{\theta_\scr{M}} & \scr{N}\ar[d]^{\theta_\scr{N}}\\
R^{(1)}\otimes_R\scr{M} \ar[r]^{{\bf 1}\otimes \varphi} & R^{(1)}\otimes_R\scr{N}
}\]
We will simply write this $F$-module morphism as $\varphi:\scr{M}\to \scr{N}$ whenever the context is clear.

\item A \emph{generating morphism} of an $F$-module is an $R$-module homomorphism $\beta: M\to \bF(M)$, where $M$ is an $R$-module, such that $\scr{M}$ is the direct limit of the top row of the following commutative diagram
\[\xymatrix{
M \ar[r]^{\beta} \ar[d]^{\beta} & \bF(M) \ar[r]^{\bF(\beta)} \ar[d]^{\bF(\beta)} & \bF^{2}(M) \ar[r] \ar[d] &\cdots\\
\bF(M) \ar[r]^{\bF(\beta)} & \bF^{2}(M) \ar[r]^{\bF^{2}(\beta)} & \bF^{3}(M)\ar[r] & \cdots
}\]
and the structure isomorphism $\theta:\scr{M}\to \bF(\scr{M})$ is induced by the vertical morphism in the diagram.

\item An $F$-module $\scr{M}$ is \emph{$F$-finite} if it admits a generating morphism $\beta:M\to \bF(M)$ where $M$ is a finitely generated $R$-module.
\item Each $F$-finite $F$-module $\scr{M}$ admits an injective generating morphism $\beta:M\hookrightarrow \bF(M)$ where $M$ is a finitely generated $R$-module; $(M,\beta)$ is called a root of $\scr{M}$.
\item For each $f\in R$, the localization $R_f$ is an $F$-finite $F$-module. 
\item Given elements $g_1,\dots,g_s\in R$, the \v{C}ech complex $\check{C}^\bullet(\underline{g};R)$ is a complex in the category of $F$-finite $F$-modules; that is, each module $\check{C}^j$ is an $F$-finite $F$-module and the differentials $\delta^j$ in this complex are $F$-module morphisms.
\item $\ker(\delta^j)$ and $\im(\delta^j)$ are $F$-finite $F$-modules and consequently $H^j_I(R)$ is an $F$-finite $F$-module for each integer $j$ and each ideal $I$ in $R$.
\end{enumerate}

Let $\scr{M}$ be an $F$-finite $F$-module and $\beta:M\hookrightarrow \bF(M)$ is a root. Let $R^{b}\xrightarrow{A} R^a\to M\to 0$ be a presentation of $M$ where $A$ is an $a\times b$ matrix whose entries are elements of $R$. Then we have the following commutative diagram:
\[
\xymatrix{
R^b \ar[r]^{A} \ar[d] & R^a \ar[r] \ar[d]^{U} &M \ar[r] \ar[d]^{\beta} & 0\\
R^b \ar[r]^{A^{[p]}} & R^a \ar[r] & F(M)\ar[r] &0
}
\]
where $A^{[p]}$ denotes the matrix whose entries are the $p$-th powers of the corresponding entries in $A$ and $U$ is an $a\times a$ matrix with entries in $R$. To ease notation, we will denote this diagram by 
\[\coker(A)\xrightarrow{U}\coker(A^{[p]}).\]

Let $f_1,\dots,f_c$ be a sequence of elements in $R$ and let $H^i(\underline{f};-)$ denote the $i$-th Koszul cohomology functor. That is, 
\[H^c(\underline{f};N)\cong N/(\underline{f}N)\quad {\rm and}\quad H^0(\underline{f};N)\cong \bigcap_{j=1}^t\ker(N\xrightarrow{f_j}N)\]
for each $R$-module $N$.

\begin{theorem}
\label{Koszul two ends closed}
For each $F$-finite $F$-module $\scr{M}$, we have that $\Supp(H^c(\underline{f};\scr{M}))$ and $\Supp(H^0(\underline{f};\scr{M}))$ are Zariski-closed, where $\underline{f}=\{f_1,\dots,f_c\}$ is an arbitrary sequence of elements in $R$.  
\end{theorem}

Before we proceed to the proof, we remark that the special case of Theorem \ref{Koszul two ends closed} when $c=1$ and $\scr{M}=H^j_I(R)$ recovers \cite[Theorem~1.1]{HochsterNB} and \cite[Theorem~7.1(c)]{KatzmanZhang}.

\begin{proof}[Proof of Theorem \ref{Koszul two ends closed}]

To treat the $0$-th Koszul cohomology, we consider the following diagram:
\begin{equation}
\label{diagram t-th Koszul}
\xymatrix@R+2pc@C+2pc{
\coker(A) \ar[r]^{U} \ar[d]_{\begin{pmatrix}f_1\\ \vdots \\f_t \end{pmatrix}} & \coker(A^{[p]})\ar[r]^{U^{[p]}} \ar[d]_{\begin{pmatrix}f_1\\ \vdots \\f_c \end{pmatrix}} & \coker(A^{[p^2]}) \ar[r]^-{U^{[p^2]}} \ar[d]_{\begin{pmatrix}f_1\\ \vdots \\f_c \end{pmatrix}} &\cdots\\
\coker(A)^{\oplus c} \ar[r]^{U^{\oplus c}} & \coker(A^{[p]})^{\oplus c} \ar[r]^{(U^{[p]})^{\oplus c}}& \coker(A^{[p^2]})^{\oplus c}\ar[r]^-{(U^{[p^2]})^{\oplus c}}  &\cdot
}
\end{equation}
Each square in this commutative diagram 
\[\xymatrix@R+2pc@C+2pc{
\coker(A^{[p^e]})\ar[r]^{U^{[p^e]} }\ar[d]_{\begin{pmatrix}f_1\\ \vdots \\f_c \end{pmatrix}} & \coker(A^{[p^{e+1}]})  \ar[d]_{\begin{pmatrix}f_1\\ \vdots \\f_c \end{pmatrix}} \\
\coker(A^{[p^e]})^{\oplus c} \ar[r]^{U^{[p^e]}}& \coker(A^{[p^{e+1}]})^{\oplus c}
}\]
commutes since $U^{[p^e]}f_j=f_jU^{[p^e]}$ for each $f_j$. Therefore (\ref{diagram t-th Koszul}) is a commutative diagram. One can check that the direct limit of (\ref{diagram t-th Koszul}) is 
\[\scr{M} \xrightarrow{\begin{pmatrix}f_1\\ \vdots \\f_c \end{pmatrix}}\scr{M}^{\oplus c}.\]

It follows from the proof of \cite[Theorem~7.1]{KatzmanZhang} that 
\[\Supp(\ker(\scr{M}\xrightarrow{f_j}\scr{M}))=\Supp(\frac{(\ker(U^{[p^j]}\cdots U)):_{R^a}f_j}{\ker(U^{[p^j]}\cdots U)}),\ j\gg 0.\]
Consequently
\[\Supp(H^0(\underline{f};\scr{M}))=\Supp(\frac{(\ker(U^{[p^j]}\cdots U)):_{R^a}(f_1,\dots,f_c)}{\ker(U^{[p^j]}\cdots U)}),\ j\gg 0\]
which is Zariski-closed.

To handle the $t$-th Koszul cohomology, we consider the following diagram:
\begin{equation}
\label{diagram 0-th Koszul}
\xymatrix@R+2pc@C+2pc{
\coker(A)^{\oplus c} \ar[r]^{U^{\oplus c}} \ar[d]_{(f_1,\cdots, f_c)} & \coker(A^{[p]})^{\oplus c}\ar[r]^{(U^{[p]})^{\oplus c}} \ar[d]_{(f_1,\cdots, f_c)} & \coker(A^{[p^2]})^{\oplus c} \ar[r]^-{(U^{[p^2]})^{\oplus c}} \ar[d]_{(f_1,\cdots, f_c)} &\cdots\\
\coker(A) \ar[r]^{U} & \coker(A^{[p]}) \ar[r]^{U^{[p]}}& \coker(A^{[p^2]})\ar[r]^-{U^{[p^2]}}  &\cdot
}
\end{equation}

Each square in this commutative diagram 
\[\xymatrix{
\coker(A^{[p^e]})^{\oplus t}\ar[r]^{(U^{[p^e]})^{\oplus t}} \ar[d]_{(f_1,\cdots, f_c)} & \coker(A^{[p^{e+1}]})^{\oplus c}  \ar[d]_{(f_1,\cdots, f_c)} \\
\coker(A^{[p^e]}) \ar[r]^{U^{[p^e]}}& \coker(A^{[p^{e+1}]})
}\]
commutes since $U^{[p^e]}f_j=f_jU^{[p^e]}$ for each $f_j$. Therefore (\ref{diagram 0-th Koszul}) is a commutative diagram. One can check that the direct limit of (\ref{diagram 0-th Koszul}) is 
\[\scr{M}^{\oplus t} \xrightarrow{(f_1,\cdots, f_c)}\scr{M}.\]
Each element in $\scr{M}$ can be represented by an element $z\in \coker(A^{[p^e]})$ for some $e$. Let $\frakp$ be a prime ideal of $R$. This element becomes $0$ in $(H^c(\underline{f};\scr{M})_{\frakp}$ if and only if there is an integer $j$ such that 
\[(U^{[p^{e+j}]}\cdots U^{[p^e]})z\in \Big(\im((f_1,\dots,f_c))+\im(A^{[p^{e+j+1}]}) \Big).\]
Therefore, 
\[(H^c(\underline{f};\scr{M})_{\frakp}=0\Leftrightarrow \bigcup_j \Big((\im((f_1,\dots,f_c))+\im(A^{[p^{e+j+1}]})):_{R^a} (U^{[p^{e+j}]}\cdots U^{[p^e]})\Big)_{\frakp} =R^a_{\frakp},\ \forall\ e.\]
Since 
\begin{align*}
 &\quad \Big((\im((f_1,\dots,f_c))+\im(A^{[p^{e+j+1}]})):_{R^a} (U^{[p^{e+j}]}\cdots U^{[p^e]})\Big)^{[p]}\\
 &= (\im((f^p_1,\dots,f^p_c))+\im(A^{[p^{e+j+2}]})):_{R^a} (U^{[p^{e+j+1}]}\cdots U^{[p^{e+1}]})\\
 &\subseteq (\im((f_1,\dots,f_c))+\im(A^{[p^{e+j+2}]})):_{R^a} (U^{[p^{e+j+1}]}\cdots U^{[p^{e+1}]}),
\end{align*}
one can check that 
\[(H^t(\underline{f};\scr{M})_{\frakp}=0\Leftrightarrow \bigcup_j \Big((\im((f_1,\dots,f_c))+\im(A^{[p^{e+j+1}]})):_{R^a} (U^{[p^{e+j}]}\cdots U^{[p^e]})\Big)_{\frakp} =R^a_{\frakp}\]
if and only if 
\[(H^t(\underline{f};\scr{M})_{\frakp}=0\Leftrightarrow \bigcup_j \Big((\im((f_1,\dots,f_c))+\im(A^{[p^{j+1}]})):_{R^a} (U^{[p^{j}]}\cdots U)\Big)_{\frakp} =R^a_{\frakp}\ {\rm (that\ is\ when\ }e=0).\]
This proves that
\[\Supp(H^t(\underline{f};\scr{M}))= \Supp\Big( \frac{R^a}{(\im((f_1,\dots,f_c))+\im(A^{[p^{j+1}]})):_{R^a} (U^{[p^{j}]}\cdots U)}\Big)\]
which is clearly Zariski-closed.
\end{proof}

The most relevant case to this article is when $\underline{f}$ is a regular sequence in $R$. We pose the following question:
\begin{question}
\label{ques: koszul of F-finite F-modules}
Let $R$ be a noetherian regular ring of primes characteristic $p$ and $\underline{f}$ be a regular sequence in $R$. Is it true that $\Supp(H^i(K^\bullet(\underline{f};\scr{M}))$ is Zariski-closed for each integer $i$ and each $F$-finite $F$-module $\scr{M}$?
\end{question}


To the best our knowledge, Question \ref{ques: koszul of F-finite F-modules} is open as stated. In the next section, we will show that it has an affirmative answer when $\underline{f}=f_1,f_2$.

\section{Regular sequences of length $2$}
\label{1st Koszul length 2}

In this section we consider the case when $t=2$; that is, when $R$ is an $F$-finite noetherian regular ring of prime characteristic, $f_1,f_2$ form a regular seqeunce in $R$ and $\scr{M}$ is an $F$-finite $F$-module. The main goal in this section is to prove the following result:
\begin{theorem}
\label{1st Koszul with two elements}
$\Supp(H^1(K^\bullet(f_1,f_2;\scr{M})))$ is Zariski-closed for every $F$-finite $F$-module $\scr{M}$ and arbitrary elements $f_1,f_2$ in $R$.
\end{theorem}

Before we can prove Theorem \ref{1st Koszul with two elements}, we would like to consider a special case of it:
\begin{theorem}
\label{torsion case of length 2}
Assume that an $F$-finite $F$-module $\scr{M}$ is $(f_1,f_2)$-torsion. Then $\Supp(H^1(K^\bullet(f_1,f_2;\scr{M})))$ is Zariski-closed.
\end{theorem}
\begin{proof}
It follows the following long exact sequence of Koszul cohomology

\begin{align*}
0&\leftarrow H^2(K^\bullet(f_1,f_2;\scr{M}))\leftarrow H^1(K^\bullet(f_1;\scr{M}))\xleftarrow{f_2}H^1(K^\bullet(f_1;\scr{M}))\\
&\leftarrow H^1(K^\bullet(f_1,f_2;\scr{M}))\leftarrow H^0(K^\bullet(f_1;\scr{M}))\xrightarrow{f_2}H^0(K^\bullet(f_1;\scr{M}))\leftarrow H^0(K^\bullet(f_1,f_2;\scr{M}))\leftarrow 0.
\end{align*}
that
{\small
\begin{align*}
&\quad \Supp(H^1(K^\bullet(f_1,f_2;\scr{M})))\\
&=\Supp(\coker(H^0(K^\bullet(f_1;\scr{M}))\xrightarrow{f_2}H^0(K^\bullet(f_1;\scr{M}))))\bigcup \Supp(\ker(H^1(K^\bullet(f_1;\scr{M}))\xrightarrow{f_2}H^1(K^\bullet(f_1;\scr{M}))))
\end{align*}}

Note that swapping $f_1$ and $f_2$ does not affect $H^1(K^\bullet(f_1,f_2;\scr{M}))$; consequently 
\[\Supp(\coker(H^0(K^\bullet(f_2;\scr{M}))\xrightarrow{f_1}H^0(K^\bullet(f_2;\scr{M}))))\subseteq \Supp(H^1(K^\bullet(f_1,f_2;\scr{M}))).\]

Hence
\begin{align*}\Supp(H^0(K^\bullet(f_1,f_2;\scr{M})))=&\Supp(\coker(H^0(K^\bullet(f_1;\scr{M}))\xrightarrow{f_2}H^0(K^\bullet(f_1;\scr{M}))))\\
&\bigcup  \Supp(\coker(H^0(K^\bullet(f_2;\scr{M}))\xrightarrow{f_1}H^0(K^\bullet(f_2;\scr{M}))))\\
&\bigcup \Supp(\ker(H^1(K^\bullet(f_1;\scr{M}))\xrightarrow{f_2}H^1(K^\bullet(f_1;\scr{M})))).
\end{align*}

First we treat $\Supp(\ker(H^1(K^\bullet(f_1;\scr{M}))\xrightarrow{f_2}H^1(K^\bullet(f_1;\scr{M}))))$. Note that 
\[\ker(H^1(K^\bullet(f_1;\scr{M}))\xrightarrow{f_2}H^1(K^\bullet(f_1;\scr{M})))\cong \ker(\frac{\scr{M}}{f_1\scr{M}}\xrightarrow{f_2}\frac{\scr{M}}{f_1\scr{M}})\cong \frac{f_1\scr{M}:_{\scr{M}}f_2}{f_1\scr{M}}.\]
Let $L$ denote a root of $\scr{M}$; that is, $L$ is finitely generated $R$-submodule of $\scr{M}$ equipped with an injective $R$-module morphism $\beta:L\to \bF(L)$ that generates the $F$-module $\scr{M}$. We will set $L_e:=\bF^e(L)\subseteq \scr{M}$ and view $L_e$ as a submodule of $L_{e+1}$ via the injective $R$-module morphism $F^e(\beta)$. Note that $\scr{M}=\cup_{e\geq 1}L_e$.

{\it Claim 1.} $\Supp\Big(\frac{f_1\scr{M}:_{\scr{M}}f_2}{f_1\scr{M}}\Big)=\bigcup_{e\geq 1}\Supp\Big( \frac{f_1\scr{M}\cap L_e:_{L_e}f_2}{f_1\scr{M}\cap L_e}\Big)$.

Assume that $\frac{f_1\scr{M}:_{\scr{M}}f_2}{f_1\scr{M}}=0$. For each $e\geq 1$ and $z_e\in (f_1\scr{M}\cap L_e:_{L_e}f_2)$, it follows that $f_2z_e\in f_1\scr{M}\cap L_e\subseteq f_1\scr{M}$ and consequently $z_e\in f_1\scr{M}\cap L_e$. This shows that $\frac{f_1\scr{M}\cap L_e:_{L_e}f_2}{f_1\scr{M}\cap L_e}=0$ for each $e$; that is, 
\[\Supp\Big(\frac{f_1\scr{M}:_{\scr{M}}f_2}{f_1\scr{M}}\Big)\supseteq \bigcup_{e\geq 1}\Supp\Big( \frac{f_1\scr{M}\cap L_e:_{L_e}f_2}{f_1\scr{M}\cap L_e}\Big).\]
On the other hand, assume that $\frac{f_1\scr{M}\cap L_e:_{L_e}f_2}{f_1\scr{M}\cap L_e}=0$ for each $e$. For each $z\in f_1\scr{M}:_{\scr{M}}f_2\subseteq \scr{M}$, there is an $e$ such that $z\in L_e$. Consequently $f_2z\in f_1\scr{M}\cap L_e$ and hence $z\in f_1\scr{M}\cap L_e\subseteq f_1\scr{M}$ by the assumption. This shows that $\frac{f_1\scr{M}:_{\scr{M}}f_2}{f_1\scr{M}}=0$; that is,
\[\Supp\Big(\frac{f_1\scr{M}:_{\scr{M}}f_2}{f_1\scr{M}}\Big)\subseteq \bigcup_{e\geq 1}\Supp\Big( \frac{f_1\scr{M}\cap L_e:_{L_e}f_2}{f_1\scr{M}\cap L_e}\Big).\]
This finishes the proof of our Claim 1.

{\it Claim 2.} $\Supp( \frac{f_1\scr{M}\cap L:_{L}f_2}{f_1\scr{M}\cap L})= \bigcup_{e\geq 1}\Supp\Big( \frac{f_1\scr{M}\cap L_e:_{L_e}f_2}{f_1\scr{M}\cap L_e}\Big)$.

It suffices to show that if $\frac{f_1\scr{M}\cap L:_{L}f_2}{f_1\scr{M}\cap L}=0$ then $\frac{f_1\scr{M}\cap L_e:_{L_e}f_2}{f_1\scr{M}\cap L_e}=0$ for each $e\geq 1$. Applying the functor $\bF^e(-)$ to the assumption $\frac{f_1\scr{M}\cap L:_{L}f_2}{f_1\scr{M}\cap L}=0$, one deduces that $\frac{f^{p^e}_1\scr{M}\cap L_e:_{L_e}f^{p^e}_2}{f^{p^e}_1\scr{M}\cap L_e}=0$; that is,
\[f^{p^e}_1\scr{M}\cap L_e:_{L_e}f^{p^e}_2=f^{p^e}_1\scr{M}\cap L_e.\]
Let $z_e$ be an element in $f_1\scr{M}\cap L_e:_{L_e}f_2$. Since $\scr{M}$ is $(f_1,f_2)$-torsion, there exists an integer $j$ such that $f_2^{jp^e}z_e=0$. Since $f_2^{p^e}(f^{(j-1)p^e}_2z_e)=0\in f^{p^e}_1\scr{M}\cap L_e$, it follows that $f^{(j-1)p^e}_2z_e\in f^{p^e}_1\scr{M}\cap L_e$. Repeating this process, one deduces that $z_e\in f^{p^e}_1\scr{M}\cap L_e\subseteq f_1\scr{M}\cap L_e$. This proves our Claim 2.

Combining these two claims shows that 
\[\Supp(\ker(H^1(K^\bullet(f_1;\scr{M}))\xrightarrow{f_2}H^1(K^\bullet(f_1;\scr{M}))))=\Supp( \frac{f_1\scr{M}\cap L:_{L}f_2}{f_1\scr{M}\cap L})\]
which is Zariski closed as $L$ is finitely generated.

It remains to prove that 
{\small
\[\Supp(\coker(H^0(K^\bullet(f_1;\scr{M}))\xrightarrow{f_2}H^0(K^\bullet(f_1;\scr{M})))) \bigcup  \Supp(\coker(H^0(K^\bullet(f_2;\scr{M}))\xrightarrow{f_1}H^0(K^\bullet(f_2;\scr{M}))))\]}
is Zariski closed (which will complete the proof of our lemma). 

Note that 
\[H^0(K^\bullet(f_1;\scr{M}))\cong (0:_{\scr{M}}f_1)\quad {\rm and}\quad H^0(K^\bullet(f_2;\scr{M}))=(0:_{\scr{M}}f_2)\]
and consequently
\begin{align*}
\coker(H^0(K^\bullet(f_1;\scr{M}))\xrightarrow{f_2}H^0(K^\bullet(f_1;\scr{M})))&\cong \frac{(0:_{\scr{M}}f_1)}{f_2(0:_{\scr{M}}f_1)}\\
\coker(H^0(K^\bullet(f_2;\scr{M}))\xrightarrow{f_1}H^0(K^\bullet(f_2;\scr{M})))&\cong \frac{(0:_{\scr{M}}f_2)}{f_1(0:_{\scr{M}}f_2)}
\end{align*}

Since $\scr{M}=\cup_{e\geq 0}L_e$, it is straightforward to check that
\begin{equation}
\label{supp coker decomp into Le}
\begin{split}
\Supp(\frac{(0:_{\scr{M}}f_1)}{f_2(0:_{\scr{M}}f_1)})&=\bigcup_e\Supp(\frac{(0:_{L_e}f_1)}{f_2(0:_{\scr{M}}f_1)\cap (0:_{L_e}f_1)})\\
\Supp(\frac{(0:_{\scr{M}}f_2)}{f_1(0:_{\scr{M}}f_2)})&=\bigcup_e\Supp(\frac{(0:_{L_e}f_2)}{f_1(0:_{\scr{M}}f_2)\cap (0:_{L_e}f_2)})
\end{split}
\end{equation}

Since $L$ is finitely generated and is $(f_1,f_2)$-torsion, there is an integer $e_0$ such that
\begin{enumerate}
\item $f^{p^{e_0}}_1L=f^{p^{e_0}}_2L=0$, and
\item $f_1(0:_{\scr{M}}f_2)\cap (0:_Lf_2)=f_1(0:_{L_{e_0}}f_2)\cap (0:_Lf_2)$, and
\item $f_2(0:_{\scr{M}}f_1)\cap (0:_Lf_1)=f_2(0:_{L_{e_0}}f_1)\cap (0:_Lf_1)$.
\end{enumerate}

Note that $f^{p^{e_0}}_1L=f^{p^{e_0}}_2L=0$ implies that 
\begin{equation}
\label{torsion for Le}
f^{p^{e_0+e}}_1L_e=f^{p^{e_0+e}}_2L_e=0
\end{equation}
for each integer $e\geq 1$.

{\it Claim 3.} 
\begin{align*}
&\quad\Supp(\frac{(0:_{\scr{M}}f_1)}{f_2(0:_{\scr{M}}f_1)})\cup \Supp(\frac{(0:_{\scr{M}}f_2)}{f_1(0:_{\scr{M}}f_2)})\\
&=\Supp(\frac{(0:_{L}f_1)}{f_2(0:_{\scr{M}}f_1)\cap (0:_{L}f_1)})\cup \Supp(\frac{(0:_{L_{e_0}}f_1)}{f_2(0:_{\scr{M}}f_1)\cap (0:_{L_{e_0}}f_1)})\\
&\quad \cup  \Supp(\frac{(0:_{L}f_2)}{f_1(0:_{\scr{M}}f_2)\cap (0:_{L}f_2)}) \cup \Supp(\frac{(0:_{L_{e_0}}f_2)}{f_1(0:_{\scr{M}}f_2)\cap (0:_{L_{e_0}}f_2)})
\end{align*}

The inclusion $\supseteq$ follows from (\ref{supp coker decomp into Le}); it remains to show $\subseteq$. To this end, assume that
\begin{itemize}
\item $(0:_{L}f_1)\subseteq f_2(0:_{\scr{M}}f_1)$, and
\item $(0:_{L_{e_0}}f_1)\subseteq f_2(0:_{\scr{M}}f_1)$, and
\item $(0:_{L}f_2)\subseteq f_1(0:_{\scr{M}}f_2)$, and
\item $(0:_{L_{e_0}}f_2)\subseteq f_1(0:_{\scr{M}}f_2)$.
\end{itemize}
and we need to show $(0:_{\scr{M}}f_1)=f_2(0:_{\scr{M}}f_1)$ and $(0:_{\scr{M}}f_2)=f_1(0:_{\scr{M}}f_2)$. 

Note it follows from our choice of $e_0$ that $(0:_{L}f_1)\subseteq f_2(0:_{L_{e_0}}f_1)$ and $(0:_{L}f_2)\subseteq f_1(0:_{L_{e_0}}f_2)$.

Given the symmetry between $f_1$ and $f_2$, it suffices to show that $(0:_{\scr{M}}f_1)=f_2(0:_{\scr{M}}f_1)$. 

Let $z\in (0:_{\scr{M}}f_1)$ be an arbitrary nonzero element. Then $z\in (0:_{L_e}f_1)$ for an integer $e$ since $\scr{M}=\cup_eL_e$. It follows from (\ref{torsion for Le}) that $f^{p^{e_0+e}}_2z=0$ since $f^{p^{e_0+e}}_2L_e=0$. That is, 
\[z\in (0:_{L_e}f^{p^{e_0+e}}_2)\subseteq (0:_{L_{e_0+e}}f^{p^{e_0+e}}_2)=\bF^{e_0+e}(0:_Lf_2)\subseteq \bF^{e_0+e}(f_1(0:_{L_{e_0}}f_2))=f^{p^{e_0+e}}_1(0:_{L_{2e_0+e}}f^{e_0+e}_2))\]
Hence, there is a $y\in (0:_{L_{2e_0+e}}f^{e_0+e}_2))$ such that $z=f^{p^{e_0+e}}_1y=f^{p^{e_0+e}-1}_1(f_1y)$. Note that
\[f^{p^{e_0+e}}_1(f_1y)=f_1f^{p^{e_0+e}}_1y=f_1z=0\]
which implies that
\[f_1y\in (0:_{L_{2e_0+e}}f^{p^{e_0+e}}_1)=\bF^{e_0+e}((0:_{L_{e_0}}f_1))\subseteq \bF^{e_0+e}(f_2(0:_{\scr{M}}f_1))=f^{p^{e_0+e}}_2(0:_{\scr{M}}f^{p^{e_0+e}}_1)\]
Thus, there is an $w\in (0:_{\scr{M}}f^{p^{e_0+e}}_1)$ such that $f_1y=f^{p^{e_0+e}}_2w$. Set 
\[x=f^{p^{e_0+e}-1}_1f^{p^{e_0+e}-1}_2w.\]
Then 
\[f_2x=f_2f^{p^{e_0+e}-1}_2f^{p^{e_0+e}-1}_1x=f^{p^{e_0+e}-1}_1 f^{p^{e_0+e}}_2w=f^{p^{e_0+e}-1}_1f_1y=f^{p^{e_0+e}}_1y=z\]
and
\[f_1x=f_1f^{p^{e_0+e}-1}_2f^{p^{e_0+e}-1}_1w=f^{p^{e_0+e}-1}_2f^{p^{e_0+e}}_1w=0\]
since $f^{p^{e_0+e}}_1w=0$ by the choice of $w$. This proves that $z=f_2x$ and $x\in (0:_{\scr{M}}f_1)$; that is, $z\in f_2(0:_{\scr{M}}f_1)$ and hence completes the proof of our Claim 3.

Note that Claim 3 implies $\Supp(\frac{(0:_{\scr{M}}f_1)}{f_2(0:_{\scr{M}}f_1)})\cup \Supp(\frac{(0:_{\scr{M}}f_2)}{f_1(0:_{\scr{M}}f_2)})$ is Zariski closed since both $L$ and $L_{e_0}$ are finitely generated.

Combining our 3 claims completes the proof of our theorem.
\end{proof}

We now return to the general case when $\scr{M}$ is an arbitrary $F$-finite $F$-module. Let $\Gamma$ denote $\Gamma_{(f_1,f_2)}(\scr{M})$. The short exact sequence
\[0\to \Gamma\to \scr{M}\to \scr{M}/\Gamma \to 0\]
induces an exact sequence on Koszul cohomology
\begin{equation}
\label{koszul involving gamma}
0= H^0(K^\bullet(\underline{f};\scr{M}/\Gamma(\scr{M})))\to H^1(K^\bullet(\underline{f};\Gamma))\to H^1(K^\bullet(\underline{f};\scr{M}))\to H^1(K^\bullet(\underline{f};\scr{M}/\Gamma))\xrightarrow{\delta}H^2(K^\bullet(\underline{f};\Gamma))
\end{equation}

The connecting morphism $\delta$ can be constructed as follows. Each element in $H^1(\underline{f};\scr{M}/\Gamma)$ can be represented by a pair $(a,b)$ with $-f_2a+f_1b=0\in \scr{M}/\Gamma$ and $a,b\in \scr{M}/\Gamma$; equivalently, each element in $H^1(\underline{f};\scr{M}/\Gamma)$ can be represented by a pair $(a,b)$ in $\scr{M}\oplus \scr{M}$ such that $-f_2 a+f_1b\in \Gamma$. Then
\[\delta(a,b)=\overline{-f_2 a+f_1b}\in \frac{\Gamma}{(f_1,f_2)\Gamma}\cong H^2(K^\bullet(\underline{f};\Gamma)).\]

Following notation in the proof of Lemma \ref{torsion case of length 2}, we denote by $L$ a root of $\scr{M}$; that is, $L$ is a finitely generated $R$-module with an injective $R$-module morphism $\beta:L\to \bF(L)$ that generates $\scr{M}$.

\begin{lemma}
\label{kernel of delta}
$\Supp(\ker(\delta))=\Supp\Big(\frac{(f_1\scr{M}\cap L:_L f_2)}{(f_1\scr{M}\cap L:_L f_2)\cap (\cup_{j\geq 0}((f^{j+1}_1\scr{M}\cap L:_L f^j_1)))} \Big)$. In particular, it is Zariski closed.
\end{lemma}
\begin{proof}
First we would like to prove that following claim.

{\it Claim.} $\Supp(\ker(\delta))=\Supp\Big(\frac{(f_1\scr{M}:_{\scr{M}}f_2)}{(f_1\scr{M}:_{\scr{M}}f_2)\cap (\cup_{j}(f^{j+1}_1\scr{M}:_{\scr{M}}f^j_1))} \Big)$.

To prove our claim, we show that 
\[\ker(\delta)=0\Leftrightarrow (f_1\scr{M}:_{\scr{M}}f_2)=(f_1\scr{M}:_{\scr{M}}f_2)\cap (\cup_{j}(f^{j+1}_1\scr{M}:_{\scr{M}}f^j_1)).\]

Each element in $\ker(\delta)$ can be represented by $(a,b)$ with $a,b\in \scr{M}$ such that $f_1b-f_2a\in (f_1,f_2)\Gamma$. That is, there are $u,v\in \Gamma$ such that $f_2b-f_1a=f_1u+f_2v$. By replacing $a,b$ with $a+u,b-v$ (which does not change the images of $a,b$ in $\scr{M}/\Gamma$), one can assume that $f_2a=f_1b$. 

Assume that $\ker(\delta)=0$. Given each $a\in (f_1\scr{M}:_{\scr{M}}f_2)$, there is an element $b\in \scr{M}$ such that $f_2a=f_1b$ and hence $(a,b)$ produces an element in $\ker(\delta)$ is zero by our assumption. Hence there is an element $c\in \scr{M}$ such that
\[(f_1c,f_2c)=(a,b)\in (\scr{M}/\Gamma)^{\oplus 2};\] 
that is, there is an integer $j$ such that $f^j_1(f_1c-a)=0$ which implies that $a\in (f^{j+1}_1\scr{M}:_{\scr{M}}f^j_1)$. This proves that $(f_1\scr{M}:_{\scr{M}}f_2)=(f_1\scr{M}:_{\scr{M}}f_2)\cap (\cup_{j}(f^{j+1}_1\scr{M}:_{\scr{M}}f^j_1))$.

On the other hand, assume that $(f_1\scr{M}:_{\scr{M}}f_2)=(f_1\scr{M}:_{\scr{M}}f_2)\cap (\cup_{j}(f^{j+1}_1\scr{M}:_{\scr{M}}f^j_1))$. Let $(a,b)$ be an element in $\ker(\delta)$. According to the discussion above, we can assume that $f_2a=f_1b$ and hence $a\in (f_1\scr{M}:_{\scr{M}}f_2)$. It follows from the assumption that there is an integer $j$ such that $f_1^ja=f^{j+1}_1c$. Then
\[f^{j+1}_1(f_2c-b)=f_2f_1^{j+1}a-f^{j+1}_1b=f_2f^j_1a-f^{j+1}_1b=f^{j+1}_1b-f^{j+1}_1b=0\]
and hence
\[(f_1c,f_2c)=(a,b)\in (\scr{M}/\Gamma)^{\oplus 2}\]
which shows that $(a,b)=0\in H^1(\underline{f};\scr{M}/\Gamma)$. This finishes the proof of our claim.

It remains to show that 
{\small
\[\Supp\Big(\frac{(f_1\scr{M}:_{\scr{M}}f_2)}{(f_1\scr{M}:_{\scr{M}}f_2)\cap (\cup_{j}(f^{j+1}_1\scr{M}:_{\scr{M}}f^j_1))} \Big)=\Supp\Big(\frac{(f_1\scr{M}\cap L:_L f_2)}{(f_1\scr{M}\cap L:_L f_2)\cap (\cup_{j\geq 0}((f^{j+1}_1\scr{M}\cap L:_L f^j_1)))} \Big)\]}
which is equivalent to proving
\[(f_1\scr{M}:_{\scr{M}}f_2)\subseteq \cup_{j\geq 0}(f^{j+1}_1\scr{M}:_{\scr{M}}f^j_1)\Leftrightarrow (f_1\scr{M}\cap L:_L f_2)\subseteq \cup_{j\geq 0}(f^{j+1}_1\scr{M}\cap L:_L f^j_1)\]

We begin with the implication $\Rightarrow$. Assume that $(f_1\scr{M}:_{\scr{M}}f_2)\subseteq \cup_{j\geq 0}(f^{j+1}_1\scr{M}:_{\scr{M}}f^j_1)$. Let $a\in (f_1\scr{M}\cap L:_L f_2)$ be an arbitrary element. Then, as $L\subseteq \scr{M}$, there is an integer $j$ and element $c\in \scr{M}$ such that $f^j_1a=f^{j+1}_1c$. This shows that $a\in (f^{j+1}_1\scr{M}\cap L:_L f^j_1)$ since $f^j_1a=f^{j+1}_1c\in f^{j+1}_1\scr{M}\cap L$. This proves the implication $\Rightarrow$.

We now prove the implication $\Leftarrow$. Assume that $(f_1\scr{M}\cap L:_L f_2)\subseteq \cup_{j\geq 0}(f^{j+1}_1\scr{M}\cap L:_L f^j_1)$. Let $a\in (f_1\scr{M}:_{\scr{M}}f_2)$ be an arbitrary element. Then $f_2a=f_1b$ for some element $b\in \scr{M}$. Since $\scr{M}=\cup_{e\geq 0}L_e$, there is an integer $e$ such that $a\in L_e$.

Apply the functor $\bF^e(-)$ to $(f_1\scr{M}\cap L:_L f_2)\subseteq \cup_{j\geq 0}(f^{j+1}_1\scr{M}\cap L:_L f^j_1)$. Let $a\in (f_1\scr{M}:_{\scr{M}}f_2)$ implies that
\[(f^{p^e}_1\scr{M}\cap L_e:_{L_e} f^{p^e}_2)\subseteq \cup_{j\geq 0}(f^{(j+1)p^e}_1\scr{M}\cap L_e:_{L_e} f^{jp^e}_1).\]

The equation $f_2a=f_1b$ implies that $f^{p^e})_2f^{p^e-1}_1a=f^{p^e}_1f^{p^e-1}_2b$ and hence 
\[f^{p^e-1}_1a\in (f^{p^e}_1\scr{M}\cap L_e:_{L_e} f^{p^e}_2)\subseteq \cup_{j\geq 0}(f^{(j+1)p^e}_1\scr{M}\cap L_e:_{L_e} f^{jp^e}_1).\]
Therefore, there is an integer $ell$ and element $c\in \scr{M}$ such that
\[f^{(\ell+1)p^e-1}_1a=f^{\ell p^e}_1f^{p^e-1}_1a=f^{(\ell+1)p^e}_1c\]
which implies that
\[a\in (f^{(\ell+1)p^e}_1\scr{M}:_{\scr{M}}f^{(\ell+1)p^e-1}_1)\subseteq \cup_{j\geq 0}(f^{j+1}_1\scr{M}:_{\scr{M}}f^j_1).\]
This proves the implication $\Leftarrow$ and hence finishes the proof of our lemma.
\end{proof}

\begin{proof}[Proof of Theorem \ref{1st Koszul with two elements}]
It follows from the exact sequence (\ref{koszul involving gamma}) that
\[\Supp(H^1(K^\bullet(f_1,f_2; \scr{M})))=\Supp(H^1(\underline{f};\Gamma))\cup \Supp(\ker(\delta)).\]
Combining Theorem \ref{torsion case of length 2} and Lemma \ref{kernel of delta} completes the proof. 
\end{proof}

Combining Theorems \ref{Koszul two ends closed} and \ref{1st Koszul with two elements}, the following result is immediate:
\begin{theorem}
\label{closed for Koszul on two elements}
Let $R$ be a noetherian regular ring of prime characteristic $p$ and $f_1,f_2\in R$ form a regular sequence. Then, for every $F$-finite $F$-module, $\Supp(H^i(K^\bullet(f_1,f_2; \scr{M}))$ is Zariski-closed for each integer $i$.
\end{theorem}

\section{The support of $E^{\bullet,\bullet}_{\infty}$ when $t=2$}
\label{support of E infinity}
In this section, we prove that the support of $E^{i,j}_{\infty}$ is Zariski closed for all integers $i,j$ and the main theorem of this article: Theorem \ref{closed for lc of ci}. Let $R$ be a noetherian commutative ring, $I=(g_1,\dots,g_s)$ be an ideal and $f_1,f_2\in R$ be a regular sequence. Then the Koszul (co)complex $K^\bullet(\underline{f};R)$ and the \v{C}ech complex $\check{C}^\bullet(\underline{g};R)$ induce the double complex (\ref{KC double complex}) introduced in \S\ref{double complex and spectral sequence}. This double complex induces a spectral sequence associates whose $E^{\bullet,\bullet}_2$-page is as follows:
\[E^{i,j}_2:=H^i(K^{\bullet}(\underline{f}; H^j_I(R))\Rightarrow H^{i+j}(T^\bullet).\]

Note that when $t=2$ there is only one (potentially) nontrivial differential on the $E_2$-page:
\[d^{0,j}_2:E^{0,j}_2\to E^{2,j-1}_2\]
Consequently
\begin{equation}
\label{E infinity for KC double complex}
E^{1,j}_{\infty}=E^{1,j}_2,\quad E^{0,j}_{\infty}=E^{0,j}_3=\ker(d^{0,j}_2),\quad E^{2,j}_\infty=E^{2,j}_3=\coker(d^{0,j}_2).
\end{equation}
We have seen in \S\ref{1st Koszul length 2} that the support of $E^{1,j}_2=H^1(K^\bullet(f_1,f_2;H^j_I(R)))$ is Zariski closed. It remains to show that both $\Supp(\ker(d^{0,j}_2))$ and $\Supp(\coker(d^{0,j}_2))$ are Zariski-closed. To this end, we begin with analyzing the construction of $d^{0,j}_2$.

\begin{remark}
\label{edge map on E_2}
We would like to recall the construction of $d^{0,j}_2$; the interested reader is referred to \cite[5.1.2]{WeibelBook} for more details. In order to cover the double complexes (\ref{KC double complex}) and (\ref{apply Frob to double with truncated}), we will consider a first quadrant double complex formed by the Koszul co-complex $K^{\bullet}(\underline{t};R)$ on two elements $t_1,t_2$ and a finite complex $C^{\bullet}$ of $R$-modules (differentials in $C^{\bullet}$ will be denoted by $d^\bullet_h$):
\begin{equation}
\label{general double complex}
\xymatrix{
& 0 & 0& 0 & & 0 \\
	0\ar[r] & C^0 \ar[r]\ar[u] &{C^1} \ar[r]\ar[u] &  {C^2} \ar[r]\ar[u] & \cdots \ar[r] & {C^s}\ar[r]\ar[u] & 0  \\
	0\ar[r] & (C^0)^{\oplus 2} \ar[r]\ar[u]^{{\small \begin{pmatrix}-t_2& t_1 \end{pmatrix}}} & {(C^1)^{\oplus 2}} \ar[r]\ar[u]^{{\small \begin{pmatrix}-t_2& t_1 \end{pmatrix}}} & {(C^2)^{\oplus 2}} \ar[r]\ar[u]^{{\small \begin{pmatrix}-t_2& t_1 \end{pmatrix}}}  & \cdots \ar[r]  & {(C^s)^{\oplus 2}} \ar[r]\ar[u]^{{\small \begin{pmatrix}-t_2& t_1 \end{pmatrix}}}  & 0 \\
	0\ar[r] & C^0  \ar[r]\ar[u]^{{\small \begin{pmatrix}t_1\\t_2 \end{pmatrix}}} & {C^1} \ar[r]\ar[u]^{{\small \begin{pmatrix}t_1\\t_2 \end{pmatrix}}} & {C^2} \ar[r]\ar[u]^{{\small \begin{pmatrix}t_1\\t_2 \end{pmatrix}}} & \cdots\ar[r] & {C^s} \ar[r]\ar[u]^{{\small \begin{pmatrix}t_1\\t_2 \end{pmatrix}}} & 0 \\
	& 0\ar[u] & 0\ar[u]& 0\ar[u] & & 0\ar[u]
}
\end{equation}

Each element $[\eta]\in H^0(K^\bullet(t_1,t_2;H^j(C^{\bullet}))$ is an element $[\eta]\in H^j(C^{\bullet})$ such that $(t_1[\eta], t_2[\eta])=(0,0)\in (H^j(C^{\bullet}))^{\oplus 2}$; equivalently $[\eta]$ can be represented by element $\eta\in C^j$ such that $d^j_h(\eta)=0$ and there are elements $(\alpha_1,\alpha_2)\in (C^{j-1})^{\oplus 2}$ such that 
\[d^{j-1}_h(\alpha_1)=t_1\eta\quad {\rm and}\quad d^{j-1}_h(\alpha_2)=t_2\eta.\] 
Consider $-t_2\alpha_1+t_1\alpha_2\in C^{j-1}$. Since
\[d^{j-1}_h(-t_2\alpha_1+t_1\alpha_2)=-t_2d^{j-1}_h(\alpha_1)+t_1d^{j-1}_h(\alpha_2)=-t_2t_1\eta+t_1t_2\eta=0\]
the element $-t_2\alpha_1+t_1\alpha_2\in C^{j-1}$ represents an element $[-t_2\alpha_1+t_1\alpha_2]\in H^{j-1}(C^\bullet)$. Then
\[d^{0,j}_2([\eta])=\overline{[-t_2\alpha_1+t_1\alpha_2]}\in E^{2,j-1}_2=H^2(K^\bullet(f_1,f_2;H^{j-1}(C^\bullet)))\cong \frac{H^{j-1}(C^\bullet)}{(t_1,t_2)H^{j-1}(C^\bullet)}.\]

For instance, the edge map in the spectral sequence associated with the double complex (\ref{apply Frob to double with truncated})
\[\varphi^{0,j}_{2,e}:H^0(K^\bullet(f^{p^e}_1,f^{p^e}_2;H^j(\check{C}^\bullet(\underline{g})_e))\to H^2(K^\bullet(f^{p^e}_1,f^{p^e}_2;H^{j-1}(\check{C}^\bullet(\underline{g})_e))\]
can be described as follows. Each element $[\eta]\in H^0(K^\bullet(f^{p^e}_1,f^{p^e}_2;H^j(\check{C}^\bullet(\underline{g})_e))$ is an element $[\eta]\in H^j(\check{C}^\bullet(\underline{g})_e)$ such that $(f^{p^e}_1[\eta], f^{p^e}_2[\eta])=(0,0)\in (H^j(\check{C}^\bullet(\underline{g})_e))^{\oplus 2}$; equivalently $[\eta]$ can be represented by element $\eta\in \check{C}^j(\underline{g})_e$ such that $\delta^j(\eta)=0$ and there are elements $\alpha_1,\alpha_2\in \check{C}^{j-1}(\underline{g})_e$ such that 
\[\delta^{j-1}(\alpha_1)=f^{p^e}_1\eta\quad {\rm and}\quad \delta^{j-1}(\alpha_2)=f^{p^e}_2\eta.\] 
Consider $-f^{p^e}_2\alpha_1+f^{p^e}_1\alpha_2\in C^{j-1}_e$. Since
\[\delta^{j-1}(-f^{p^e}2\alpha_1+f^{p^e}_1\alpha_2)=-f^{p^e}_2\delta^{j-1}(\alpha_1)+f^{p^e}_1\delta^{j-1}(\alpha_2)=-f^{p^e}_2f^{p^e}_1\eta+f^{p^e}_1f^{p^e}_2\eta=0\]
the element $-f^{p^e}_2\alpha_1+f^{p^e}_1\alpha_2\in \check{C}^{j-1}(\underline{g})$ represents an element $[-f^{p^e}_2\alpha_1+f^{p^e}_1\alpha_2]\in H^{j-1}_I(R)$. Then
\begin{equation}
\label{edge map Frob applied to truncated double complex}
\varphi^{0,j}_{2,e}([\eta])=\overline{[-f^{p^e}_2\alpha_1+f^{p^e}_1\alpha_2]}\in H^2(K^\bullet(f^{p^e}_1,f^{p^e}_2;H^{j-1}(\check{C}^\bullet_e)))\cong \frac{H^{j-1}(\check{C}^\bullet(\underline{g})_e)}{(f^{p^e}_1,f^{p^e}_2)H^{j-1}(\check{C}^\bullet(\underline{g})_e)}.
\end{equation}
\end{remark}

To ease notation, for the rest of this section we will denote the \v{C}ech complex $\check{C}^\bullet(\underline{g})$ by $\check{C}^\bullet$ and its $e$-th truncation $\check{C}^\bullet(\underline{g})_e$ by $\check{C}^\bullet_e$.

Recall that the double complex ${\bf D}_0$ induces the spectral sequence (\ref{spectral sequence truncated}):
\[
E^{i,j}_{2,0}:=H^i(K^{\bullet}(\underline{f}; H^j(\check{C}^\bullet_0))\Rightarrow H^{i+j}(T^\bullet_0)
\]
with the differentials 
\[\varphi^{i,j}_{2,0}:H^0(K^{\bullet}(\underline{f}; H^j(\check{C}^\bullet_0))\to H^2(K^{\bullet}(\underline{f}; H^{j-1}(\check{C}^\bullet_0)).\]
Let $K_0^j\subseteq \ker(\delta^j_0)\subseteq \check{C}^j_0$ be the submodule whose image in $H^j(\check{C}^{\bullet})$ is the kernel of $\varphi^{0,j}_2$, where $\delta^j_0$ denotes the $j$-th differential in $\check{C}^{\bullet}_0$. Note that 
\begin{enumerate}
\item $K^j_0$ is a finitely generated $R$-module since $\check{C}^j_0$ is so;
\item the image of $H^j(\check{C}^\bullet_0)$ in $H^j_I(R)$ is isomorphic to $\frac{\ker(\delta^j_0)}{\ker(\delta^j_0)\cap \im(\delta^{j-1})}$, where $\delta^j$ denotes the $j$-th differential in $\check{C}^{\bullet}$; this is contained in (\ref{image of truncated Cech in LC}).
\end{enumerate}

First we treat $\Supp(E^{0,j}_{\infty})$ which is $\Supp(\ker{d^{0,j}_2})$ (\ref{E infinity for KC double complex}) and we begin with the following lemma.

\begin{lemma}
\label{K 0 induces K e}
Let $R$ be a noetherian regular ring of prime characteristic $p$. Let
$\varphi^{0,j}_{2,e}$ be defined as in (\ref{edge map Frob applied to truncated double complex}). Let $K^j_e$ be the submodule of $\ker(\delta^j_e)\subseteq \check{C}^j_e$ whose image in $H^j(\check{C}^{\bullet}_e)$ is the kernel of $\varphi^{0,j}_{2,e}$. Let $\theta:\bF^e(\check{C}^j_0)\xrightarrow{\sim}\check{C}^j_e$ denote the isomorphism in Proposition \ref{apply Frob to truncated Cech}. Then 
\[\theta(\bF^e(K^j_0))= K^j_e.\] 
\end{lemma}
\begin{proof}
This follows from the commutative diagram below and the fact $R^{(e)}$ is a faithfully flat $R$-module.
\[
\xymatrix{
R^{(e)}\otimes (\check{C}^{j-1}_0\oplus \check{C}^{j-1}_0) \ar[r]^{\sim} \ar[d]_{{\bf 1}\otimes (\delta^{j-1}_0\oplus \delta^{j-1}_0)} & \check{C}^{j-1}_e\oplus \check{C}^{j-1}_e \ar[d]^{\delta^{j-1}_e\oplus \delta^{j-1}_e}\\
R^{(e)}\otimes (\check{C}^{j}_0\oplus \check{C}^{j}_0) \ar[r]^{\sim} & \check{C}^{j}_e\oplus \check{C}^{j}_e\\
R^{(e)}\otimes \check{C}^{j}_0 \ar[r]^{\sim} \ar[u]^{{\bf 1}\otimes {\tiny \begin{pmatrix} f_1\\f_2\end{pmatrix}} }& \check{C}^{j}_e \ar[u]_{{\bf 1}\otimes {\tiny \begin{pmatrix} f^{p^e}_1\\f^{p^e}_2\end{pmatrix}}}
}
\]
\end{proof}

\begin{theorem}
\label{kernel of edge map}
Let $R$ be a noetherian regular ring of prime characteristic $p$ and let $E^{\bullet,\bullet}_2$ be the $E_2$-page of the spectral sequence associates with the double complex (\ref{KC double complex}). Then $\ker{d^{0,j}_2}=0$ if and only if $K^j_0\subseteq \im(\delta^{j-1})$; that is,
\begin{equation}
\label{support of kernel and 0 truncate}
\Supp(E^{0,j}_{\infty})=\Supp(\ker{d^{0,j}_2})=\Supp(\frac{K^j_0}{K^j_0\cap \im(\delta^{j-1})}).
\end{equation}
In particular, $\Supp(E^{0,j}_{\infty})=\Supp(\ker{d^{0,j}_2})$ is Zariski-closed.
\end{theorem}
\begin{proof}
The second statement follows from (\ref{support of kernel and 0 truncate}) since $K^j_0$ is finitely generated.

To complete the proof, it remains to show that $\ker{d^{0,j}_2}=0$ if and only if $K^j_0\subseteq \im(\delta^{j-1})$.

Assume that $d^{0,j}_2$ is injective and $[\eta]\in K^j_0$. One needs to show that $[\eta]\in \im(\delta^{j-1})$. Since $[\eta]$ belongs to $K^j_0$, its image in  $H^j(\check{C}^{\bullet}_0)$ must belong in $\ker(\varphi^{0,j}_{2,0})$. It follows that the image of $[\eta]$ in $H^j_I(R)$ must belong in $\ker(d^{0,j}_2)$. Since $d^{0,j}_2$ is injective, the image of $[\eta]$ in $H^j_I(R)$ must be $[0]$, which implies that  $[\eta]\in \im(\delta^{j-1})$. This proves the `if' statement.

Assume that $K^j_0\subseteq \im(\delta^{j-1})$; that is, if $\varphi^{0,j}_{2,0}([\eta])=[0]$, then $\eta\in \im(\delta^{j-1})$ (equivalently, the image $[\eta]$ of $\eta$ in $H^j_I(R)$ is zero). Note it follows from Lemma \ref{K 0 induces K e} that 
\[K^j_e\cong \bF^e(K^j_0)\subseteq \bF^e(\im(\delta^{j-1}))\cong \im(\delta^{j-1})\] 
where the last isomorphism follows from that fact that $\delta^{j-1}$ is a differential in the \v{C}ech complex and hence an $F$-module morphism. 

Let $[\tau]$ be an element in $\ker(d^{0,j}_2)$, it remains to show that $[\tau]=[0]\in H^j_I(R)$. Since $[\tau]\in \ker(d^{0,j}_2)$, there are elements $\tau\in \check{C}^j$ and $\alpha_1,\alpha_2\in \check{C}^{j-1}$ such that 
\[\delta^{j-1}(\alpha_1)=f_1\tau,\ \delta^{j-1}(\alpha_2)=f_2\tau,\ {\rm and\ } d^{0,j}_2([\tau])=\overline{[-f_2\alpha_+f_1\alpha_2]}\in (f_1,f_2)H^{j-1}_I(R).\]
Since there are finitely many cohomology classes involved, there exists an integer $e$ such that $\tau\in \check{C}^j_e$, $\alpha_1,\alpha_2\in \check{C}^{j-1}_e$, and that $d^{0,j}_2([\tau])$ can be represented by an element in $(f_1,f_2)H^{j-1}(\check{C}^{\bullet}_e)$. We will fix one such $e$ and we consider the double complex (\ref{apply Frob to double with truncated}) for this integer $e$. It follows that
\[\delta^{j-1}(f^{p^e-1}_1\alpha_1)=f^{p^e}_1\tau\ {\rm and}\ \delta^{j-1}(f^{p^e-1}_2\alpha_2)=f^{p^e}_2\tau.\]
According the description of the edge map (\ref{edge map Frob applied to truncated double complex}) associated with the double complex (\ref{apply Frob to double with truncated}):
\begin{align*}
\varphi^{0,j}_{2,e}([\tau])&=\overline{[-f^{p^e}_2f^{p^e-1}_1\alpha_1+f^{p^e}_1f^{p^e-1}_2\alpha_2]}\\
&=(f^{p^e-1}_1f^{p^e-1}_2)\overline{[-f_2\alpha_+f_1\alpha_2]}\\
&\in (f^{p^e-1}_1f^{p^e-1}_2)(f_1,f_2)H^{j-1}(\check{C}^{\bullet}_e)\\
&\in  (f^{p^e}_1,f^{p^e}_2)H^{j-1}(\check{C}^{\bullet}_e)
\end{align*} 
That is $[\tau]$ belongs in $K^j_e$ and consequently $[\tau]\in K^j_e\subseteq \im(\delta^{j-1})$. Thus, the image of $[\tau]$ in $H^j_I(R)$ is zero. This shows that, if $K^j_0\subseteq \im(\delta^{j-1})$, then $d^{0,j}_2$ is injective, which completes the proof.
\end{proof}

\begin{theorem}
\label{cokernel of edge map}
Let $R$ be a noetherian regular ring of prime characteristic $p$ and let $E^{\bullet,\bullet}_2$ be the $E_2$-page of the spectral sequence associates with the double complex (\ref{KC double complex}). Let $H\subseteq H^{j-1}_I(R)$ be the submodule generated by elements that can be represented by elements in $\check{C}^{j-1}_0$. Let $L\subseteq H^{j-1}_I(R)$ be the submodule whose image in $H^{j-1}_I(R)/(f_1,f_2)H^{j-1}_I(R)$ is $\im(d^{0,j}_2)$. Then $d^{0,j}_2$ is surjective if and only if $H\subseteq L$; that is
\begin{equation}
\label{support of cokernel and 0 truncate}
\Supp(E^{2,j-1}_{\infty})=\Supp(\coker{d^{0,j}_2})=\Supp(\frac{H}{H\cap L}).
\end{equation}
In particular, $\Supp(E^{2,j-1}_{\infty})=\Supp(\coker{d^{0,j}_2})$ is Zariski-closed.
\end{theorem}
\begin{proof}
Since $H$ is finitely generated (\ref{image of truncated Cech in LC}), the Zariski-closedness follows from the `if and only if' statement.

If $d^{0,j}_2$ is surjective, then $L=H^{j-1}_I(R)$ and hence $H\subseteq L$.

Assume that $H\subseteq L$. Then $\bF^e(H)\subseteq \bF^e(L)$ for each $e$ since $\bF$ is an exact functor. Note that $\bF^e(H)$ is the submodule of $H^{j-1}_I(R)$ generated by elements that can be represented by elements in $\check{C}^{j-1}_e$ and that $\bF^e(L)$ is the submodule of $H^{j-1}_I(R)$ whose image in $H^{j-1}_I(R)/(f^{p^e}_1,f^{p^e}_2)H^{j-1}_I(R)$ is $\im(\varphi^{0,j}_{2,e})$. 

Let $[\eta]$ be an arbitrary element in $H^{j-1}_I(R)/(f_1,f_2)H^{j-1}_I(R)$. Pick an element $\eta_e$ in $\check{C}^{j-1}_e$ whose image in $H^{j-1}_I(R)/(f_1,f_2)H^{j-1}_I(R)$ is $[\eta]$. Then $[\eta_e]\in H^{j-1}_I(R)$ belongs to $\bF^e(H)$. Hence $[\eta_e]\in \bF^e(L)$; that is, there are $\tau_e\in \check{C}^j_e$, $\alpha_{1,e},\alpha_{2,e}\in \check{C}^{j-1}_e$ and $\beta_{1,e},\beta_{2,e}\in \ker(\delta^j_e)$ such that
\[\delta^j_e(\tau_e)=0,\ \delta^{j-1}_e(\alpha_{1,e})=f^{p^e}_1\tau_e,\ \delta^{j-1}_e(\alpha_{2,e})=f^{p^e}_2\tau_e\]
and that
\begin{align*}
[\eta_e]&=\varphi^{0,j}_{2,e}([\tau_e])\\
&=\overline{[-f^{p^e}_2\alpha_{1,e}+f^{p^e}_1\alpha_{2,e}]}+f^{p^e}_1\beta_{1,e}+f^{p^e}_2\beta_{2,e}\\
&=\overline{[-f_2(f^{p^e-1}_2\alpha_{1,e})+f_1(f^{p^e-1}_1\alpha_{2,e})]}+f_1(f^{p^e-1}_1\beta_{1,e})+f_2(f^{p^e-1}_2\beta_{2,e})
\end{align*}
Set $\widetilde{\tau}=f^{p^e-1}_1f^{p^e-1}_2\tau_e$, $\widetilde{\alpha}_1=f^{p^e-1}_2\alpha_{1,e}$ and $\widetilde{\alpha}_2=f^{p^e-1}_1\alpha_{2,e}$. Then
\[\delta^j_e(\widetilde{\tau})=0,\ \delta^{j-1}_e(\widetilde{\alpha}_{1})=f_1\widetilde{\tau},\ \delta^{j-1}_e(\widetilde{\alpha}_{2})=f_2\widetilde{\tau}\]
and 
\begin{align*}
[\eta_e]&=\overline{[-f_2(f^{p^e-1}_2\alpha_{1,e})+f_1(f^{p^e-1}_1\alpha_{2,e})]}+f_1(f^{p^e-1}_1\beta_{1,e})+f_2(f^{p^e-1}_2\beta_{2,e})\\
&=\overline{[-f_2\widetilde{\alpha}_1+f_1\widetilde{\alpha}_2]} +f_1(f^{p^e-1}_1\beta_{1,e})+f_2(f^{p^e-1}_2\beta_{2,e})\\
&=d^{0,j}_2(\widetilde{\tau}) 
\end{align*}
This proves that $[\eta_e]$ is in the image of $d^{0,j}_2$. This completes the proof.
\end{proof}

Combining Theorems \ref{closedness in one spectral implies the other}, \ref{closed for Koszul on two elements}, \ref{kernel of edge map}, and \ref{cokernel of edge map}, the following theorem is immediate:
\begin{theorem}
\label{closed for lc of ci}
Let $R$ be a noetherian regular ring of prime characteristic $p$. If $f_1,f_2\in R$ form a regular sequence in $R$, then
\[\Supp(H^j_I(\frac{R}{(f_1,f_2)}))\]
is Zariski-closed for each  ideal $I$ and each integer $j$.  
\end{theorem}

The following corollary is immediate.
\begin{corollary}
Let $R$ be a noetherian commutative ring of prime characteristic $p$ that has finitely many isolated singular points. Let $f_1,f_2\in R$ be a regular sequence. Then $H^j_I(R/(f_1,f_2))$ is Zariski-closed for each integer $j$ and each ideal $I$.
\end{corollary}

\end{document}